\documentclass[12pt,reqno]{amsart}
\usepackage[margin=1in]{geometry}
\makeatletter
\def\section{\@startsection{section}{1}%
	\z@{.7\linespacing\@plus\linespacing}{.5\linespacing}%
	{\bfseries
		\centering
}}
\def\@secnumfont{\bfseries}
\makeatother
\usepackage{amsmath,amssymb,amsthm,graphicx,amsxtra, setspace}
\usepackage[utf8]{inputenc}
\usepackage{charter}
\usepackage{mathrsfs}
\usepackage{alltt}
\usepackage{relsize}
\usepackage{hyperref}
\usepackage{aliascnt}
\usepackage{tikz}
\usepackage{mathtools}
\usepackage{multicol}
\usepackage{upgreek}
\usepackage{graphicx,type1cm,eso-pic,color}
\allowdisplaybreaks
\usepackage{scalerel,stackengine}
\stackMath
\newcommand\reallywidehat[1]{%
	\savestack{\tmpbox}{\stretchto{%
			\scaleto{%
				\scalerel*[\widthof{\ensuremath{#1}}]{\kern-.6pt\bigwedge\kern-.6pt}%
				{\rule[-\textheight/2]{1ex}{\textheight}}
			}{\textheight}%
		}{0.5ex}}%
	\stackon[1pt]{#1}{\tmpbox}%
}
\newtheorem{theorem}{Theorem}[section]

\newaliascnt{lemma}{theorem}
\newtheorem{lemma}[lemma]{Lemma}
\aliascntresetthe{lemma} 

\newaliascnt{proposition}{theorem}

\aliascntresetthe{proposition}

\newaliascnt{assumption}{theorem}
\newtheorem{assumption}[assumption]{Assumption}
\aliascntresetthe{assumption}

\newaliascnt{corollary}{theorem}

\aliascntresetthe{corollary}

\newaliascnt{definition}{theorem}
\newtheorem{definition}[definition]{Definition}
\aliascntresetthe{definition}

\newaliascnt{example}{theorem}
\newtheorem{example}[example]{Example}
\aliascntresetthe{example}

\newaliascnt{remark}{theorem}
\newtheorem{remark}[remark]{Remark}
\aliascntresetthe{remark}

\newaliascnt{hypothesis}{theorem}

\aliascntresetthe{hypothesis}

\newaliascnt{property}{theorem}

\aliascntresetthe{property}

\let\originalleft\left
\let\originalright\right
\renewcommand{\left}{\mathopen{}\mathclose\bgroup\originalleft}
\renewcommand{\right}{\aftergroup\egroup\originalright}

\makeatletter
\newcommand{\doublewidetilde}[1]{{%
		\mathpalette\double@widetilde{#1}%
}}
\newcommand{\double@widetilde}[2]{%
	\sbox\z@{$\m@th#1\widetilde{#2}$}%
	\ht\z@=.9\ht\z@
	\widetilde{\box\z@}%
}
\makeatother

\renewcommand{\d}{\/\mathrm{d}\/}

\def\w{\textbf{W}^{\varepsilon}_{{\theta}^{\varepsilon}}}

\def\A{\mathrm{A}}

\def\F{\mathrm{F}}
\def\C{\mathrm{C}}
\def\f{\mathbf{f}}

\def\J{\mathrm{J}}
\def\B{\mathrm{B}}
\def\D{\mathrm{D}}

\def\X{\mathbb{X}}

\def\z{\mathbf{z}}
\def\v{\mathbf{v}}
\def\V{\mathbb{V}}
\def\w{\mathbf{w}}
\def\W{\mathrm{W}}
\def\G{\mathbb{G}}

\def\V{\mathbb{V}}

\def\U{\mathrm{U}}

\def\u{\mathbf{u}}
\def\H{\mathbb{H}}

\def\p{\mathbf{p}}

\newcommand{\R}{\mathbb{R}}

\renewcommand{\d}{\/\mathrm{d}\/}

\newcommand{\Addresses}{{
		\footnote{

			\noindent \textsuperscript{1}School of Mathematics,
			Indian Institute of Science Education and Research, Trivandrum (IISER-TVM),
			Maruthamala PO, Vithura, Thiruvananthapuram, Kerala, 695 551, INDIA.  \par\nopagebreak \noindent
			\textit{e-mail:} \texttt{tania9114@iisertvm.ac.in}

			\noindent \textsuperscript{2}School of Mathematics, Indian Institute of Science Education and Research, Trivandrum (IISER-TVM),
			Maruthamala PO, Vithura, Thiruvananthapuram, Kerala, 695 551, INDIA  \par\nopagebreak \noindent
			\textit{e-mail:} \texttt{sheetal@iisertvm.ac.in}

			\noindent \textsuperscript{3}Department of Mathematics, Indian Institute of Technology (IIT), Roorkee,
			Haridwar Highway, Roorkee,  Uttarakhand 247667, INDIA \par\nopagebreak
			\noindent  \textit{e-mail:} \texttt{maniltmohan@gmail.com, manil$\_$vs@isibang.ac.in}
			
			\noindent \textsuperscript{*}Corresponding author.

			\medskip\noindent
			{\bf Acknowledgments:}  Tania Biswas  would like to thank the  Indian Institute of Science Education and Research, Thiruvananthapuram, for providing financial support and stimulating environment for the research. M. T. Mohan would like to thank the Department of Science and Technology (DST), India for Innovation in Science Pursuit for Inspired Research (INSPIRE) Faculty Award (IFA17-MA110) and Indian Institute of Technology (IIT), Roorkee, for providing stimulating scientific environment and resources.
			
}}}

\begin{document}
	
	\title[Second order optimality conditions]{Second order optimality conditions for optimal control problems governed by 2D nonlocal Cahn-Hillard-Navier-Stokes equations	\Addresses	}

	\author[T. Biswas, S. Dharmatti and M. T. Mohan ]
	{Tania Biswas\textsuperscript{1}, Sheetal Dharmatti\textsuperscript{2*} and Manil T. Mohan\textsuperscript{3}}

	\maketitle

	\begin{abstract}
		In this paper, we formulate a distributed optimal control problem related to the evolution of two isothermal, incompressible, immisible fluids in a two dimensional  bounded domain. The distributed optimal control problem is framed as the minimization of a suitable cost functional subject to the controlled nonlocal Cahn-Hilliard-Navier-Stokes equations.   We describe the first order necessary conditions of optimality via Pontryagin's minimum principle and prove second order necessary and sufficient conditions of optimality for the problem.
	\end{abstract}

	

	\section{Introduction}\label{sec1}\setcounter{equation}{0} Optimal  control theory of fluid dynamic models has been one of the important areas of applied mathematics with  several applications in engineering  and technology (see for example \cite{sritharan,fursikov,gunzburger}). The mathematical developments in infinite dimensional nonlinear system theory and partial differential equations opened up a new dimension for the optimal control theory of fluid dynamic models.
	In this work, we prove a second order necessary and sufficient conditions of optimality for an optimal control problem governed by Cahn-Hilliard-Navier-Stokes system in a two dimensional  bounded domain. We have considered a general cost functional and obtained the first and second order  necessary and sufficient optimality conditions. The first order necessary condition in the form of Pontryagin's maximum principle is established in \cite{ControlCHNS,BDM}.

	\emph{Cahn-Hilliard-Navier-Stokes system} describes the evolution of an incompressible, isothermal mixture of two immiscible fluids in a domain $ \Omega\subset\R^2 $ or $ \R^3$.   It is a coupled system with two unknowns: the average velocity of the fluid which is denoted by $\u(x,t)$ and the relative concentration of one fluid, denoted by $\varphi(x,t),$ for $(x,t)\in\Omega\times(0,T)$. The mathematical equations of the system are given by:
	
	\begin{subequations}
		\begin{align}
		\varphi_t + \u\cdot \nabla \varphi &= \Delta \mu, \label{nonlin phi}\ \text{ in }\ \Omega\times(0,T),\\
		\mu &= a \varphi - \J\ast \varphi + \F'(\varphi), \label{mu}\ \text{ in }\ \Omega\times(0,T),\\
		\u_t - \nu \Delta \u + (\u\cdot \nabla )\u + \nabla \uppi &= \mu \nabla \varphi + \mathbf{f} +\U, \label{nonlin u}\ \text{ in }\ \Omega\times(0,T),\\
		\text{div }\u&= 0, \ \text{ in }\ \Omega\times(0,T), \label{div zero}\\
		\frac{\partial \mu}{\partial \mathbf{n}} &= 0, \ \u=0, \ \text{ on } \ \partial \Omega \times (0,T),\label{boundary conditions}\\
		\u(0) &= \u_0, \  \ \varphi(0) = \varphi _0, \ \text{ in } \ \Omega. \label{initial conditions}
		\end{align}   
	\end{subequations}
	We restrict ourselves to dimension two, that is  we assume the evolution happens in,  $\Omega \subset \mathbb{R}^2$ which is a bounded domain with sufficiently smooth boundary. Here the unit outward normal to the boundary $\partial\Omega $ is  denoted by $\mathbf{n}$, $\mu$ is the \emph{chemical potential}, $\uppi$ is the \emph{pressure}, $\J$ is the \emph{spatial-dependent internal kernel}, $\J \ast \varphi$ denotes the spatial convolution over $\Omega$, $a$ is defined by $a(x) := \int _\Omega \J(x-y) \d y$, $\F$ is a double well potential, $\nu$ is the \emph{kinematic viscosity} and $\mathbf{f}$ is the \emph{external forcing} acting in the mixture. In the system (\ref{nonlin phi})-(\ref{initial conditions}),  $\U$ is the \emph{distributed control} acting in the velocity equation.
	The density is supposed to be constant and is equal to one (i.e., matched densities). The system (\ref{nonlin phi})-(\ref{initial conditions}) is  called nonlocal because of the term $\J$, which is averaged over the spatial domain. Various simplified models of this system are studied by several mathematicians and physicists.   The local version of the system  is obtained by replacing  $\mu $ equation by $ \mu = \Delta \varphi + \F' (\varphi) $. However, the nonlocal version is physically more relevant and mathematically challenging too. This model is more difficult to handle because of the nonlinear terms like the \emph{capillarity term} (i.e., \emph{Korteweg force}) $\mu \nabla \varphi$ acting on the fluid.  Even in two dimensions, this term can be less regular than the convective term $(\u\cdot\nabla)\u$ (see \cite{weak}). 
	
	Recently, the solvability of the system (\ref{nonlin phi})-(\ref{initial conditions}) has been studied extensively (see for example \cite{weak,unique,strong}, etc). Moreover, many results are available in the literature for simplified Cahn-Hilliard-Navier-Stokes models. In \cite{weak}, the authors  proved the existence of a weak solution for the system (\ref{nonlin phi})-(\ref{initial conditions}). The uniqueness of weak solution for such systems remained open until  2016 and the authors in \cite{unique} resolved it for dimension 2.   The existence of a unique strong solution in two dimensions for the system (\ref{nonlin phi})-(\ref{initial conditions}) is proved in \cite{strong}. As in the case of 3D Navier-Stokes, in three dimensions, the existence of a weak solution is known (see \cite{weak}), but the uniqueness of the weak solution  for the nonlocal Cahn-Hilliad-Navier-Stokes system remains open.

	The first order necessary conditions of optimality for various optimal control problems governed by nonlocal Cahn-Hilliad-Navier-Stokes system has been established in \cite{ControlCHNS,BDM,SFMG,BDM2}, etc. In \cite{BDM}, the authors have studied a distributed optimal control problem for nonlocal Cahn-Hilliard-Navier-Stokes system and established the Pontryagin’s maximum principle. They have characterized the optimal control using the adjoint variable. A similar kind of problem is considered in \cite{ControlCHNS}, where the authors proved the first order optimality condition under some restrictive condition on the spatial operator $\J$ (see (H4) in \cite{ControlCHNS}). A distributed optimal  control problem for two-dimensional nonlocal Cahn-Hilliard-Navier-Stokes systems with degenerate mobility and singular potential is described in \cite{SFMG}. An optimal initialization problem, a type of data assimilation problem is examined in \cite{BDM2}. A  distributed optimal control of a diffuse interface model of tumor growth is considered in \cite{tumor}. The model studied in \cite{tumor} is  a kind of local Cahn-Hilliard-Navier-Stokes type system with some additional conditions on $\F$. Optimal control problems with state constraint and robust control for local Cahn-Hilliard-Navier-Stokes system are investigated in \cite{robust, state const}, respectively. 
	
	In this paper, we formulate a distributed optimal control problem and establish a second order necessary and sufficient condition of optimality for the non local Cahn-Hilliard-Navier-Stokes system in a two dimensional bounded domain. The unique global strong solution of the system (\ref{nonlin phi})-(\ref{initial conditions}) established in \cite{strong} helps us to achieve this goal. The second order optimality condition for 3D Navier-Stokes equations in a  periodic domain is established in \cite{LijuanPezije}. The second order optimality condition for various optimization problems governed by Navier-Stokes equations is obtained in \cite{CT,TW,WD}, etc.  
	
	The paper is organized as follows: In the next section, we discuss the functional setting for the unique  solvability of the system (\ref{nonlin phi})-(\ref{initial conditions}). We also state the existence and uniqueness of a weak as well as strong solution of the system  (see Theorems \ref{exist}, \ref{unique} and \ref{strongsol}). The unique solvability of the  linearized system is stated in section \ref{se3} (see Theorem \ref{linearized}). An optimal control problem is formulated in section \ref{se4} as the minimization of a suitable cost functional. A first order necessary condition for optimality via Pontryagin's maximum principle has been established in \cite{BDM}, which is stated again in this section for completeness. We prove our main result namely the second order necessary and sufficient optimality condition for the optimal control problem in this section (see Theorems \ref{necessary} and \ref{sufficient}). We need  the existence and uniqueness of weak solutions of the linearized as well as adjoint system which has already been  proved in \cite{BDM,BDM2}.
	The final term (see \eqref{3.14}), which we observe in the second order optimality condition is due to the strong non linearity and coupling present in the equation. This term is well defined  thanks to a higher regularity established for adjoint system in  \cite{BDM2}.

	\section{Mathematical Formulation}\label{sec2}\setcounter{equation}{0} In this section, we discuss the necessary function spaces required to obtain the global solvability results for (\ref{nonlin phi})-(\ref{initial conditions}), as well as the linearized and adjoint systems. We mainly follow the papers \cite{weak,unique} for the mathematical formulation and functional setting. 
	
	\subsection{Functional setting}
	Let us introduce the following function spaces required for obtaining the unique global solvability results of the system (\ref{nonlin phi})-(\ref{initial conditions}). Note that on the boundary, we  did not prescribe any condition (Dirichlet, Neumann, etc) on the value of  $\varphi$, the relative concentration of one of the  fluid. Instead, we imposed a Neumann boundary condition for the chemical potential $\mu$ on the boundary (see \eqref{boundary conditions}). Let us define
	\begin{align*}\G_{\text{div}} &:= \Big\{ \u \in \mathrm{L}^2(\Omega;\R^2) : \text{div }\u=0, \u\cdot \mathbf{n}\big|_{\partial\Omega}=0\Big\}, \\
	\V_{\text{div}} &:= \Big\{\u \in \mathrm{H}^1_0(\Omega;\R^2): \text{div }\u=0\Big\},\\ \mathrm{H}&:=\mathrm{L}^2(\Omega;\R),\ \mathrm{V}:=\mathrm{H}^1(\Omega;\R). \end{align*}
	Let us denote $\| \cdot \|$ and $(\cdot, \cdot),$ the norm and the scalar product, respectively, on both $\mathrm{H}$ and $\G_{\text{div}}$. The norms on $\G_{\text{div}}$ and $\mathrm{H}$ are given by $\|\u\|^2:=\int_{\Omega}|\u(x)|^2\d x$ and $\|\varphi\|^2:=\int_{\Omega}|\varphi(x)|^2\d x$, respectively. The duality between $\V_{\text{div }}$ and its topological dual $\V_{\text{div}}'$ is denoted by $\langle\cdot,\cdot\rangle$. We know that $\V_{\text{div}}$ is endowed with the scalar product 
	$$(\u,\v)_{\V_{\text{div }}}= (\nabla \u, \nabla \v)=2(\mathrm{D}\u,\mathrm{D}\v),\ \text{ for all }\ \u,\v\in\V_{\text{div}},$$ where $\mathrm{D}\u$ is the strain tensor $\frac{1}{2}\left(\nabla\u+(\nabla\u)^{\top}\right)$.  The norm on $\V_{\text{div }}$ is given by $\|\u\|_{\V_{\text{div }}}^2:=\int_{\Omega}|\nabla\u(x)|^2\d x=\|\nabla\u\|^2$. In the sequel, we use the notations $\mathbb{H}^2(\Omega):=\mathrm{H}^2(\Omega;\mathbb{R}^2)$ and  $\mathrm{H}^2(\Omega):=\mathrm{H}^2(\Omega;\mathbb{R})$ for second order Sobolev spaces.

	\subsection{Weak and strong solution of the system \eqref{nonlin phi}-\eqref{initial conditions}}
	We state the results regarding the existence theorem and uniqueness of  weak and strong solutions for the uncontrolled nonlocal Cahn-Hilliard-Navier-Stokes system  given by \eqref{nonlin phi}-\eqref{initial conditions} with $\U=0$.  
	Let us first make the following assumptions:
	\begin{assumption}\label{prop of F and J} Let
		$\mathrm{J}$ and $\mathrm{F}$ satisfy:
		\begin{enumerate}
			\item [(1)] $ \J \in \W^{1,1}(\mathbb{R}^2;\R), \  \J(x)= \J(-x) \; \text {and} \ a(x) = \int\limits_\Omega \J(x-y)\d y \geq 0,$ a.e., in $\Omega$.
			\item [(2)] $\F \in \C^{2}(\mathbb{R})$ and there exists $C_0 >0$ such that $\F''(s)+ a(x) \geq C_0$, for all $s \in \mathbb{R}$, a.e., $x \in \Omega$.
			\item [(3)] Moreover, there exist $C_1 >0$, $C_2 > 0$ and $q>0$ such that $\F''(s)+ a(x) \geq C_1|s|^{2q} - C_2$, for all $s \in \mathbb{R}$, a.e., $x \in \Omega$.
			\item [(4)] There exist $C_3 >0$, $C_4 \geq 0$ and $r \in (1,2]$ such that $|\F'(s)|^r \leq C_3|\F(s)| + C_4,$ for all $s \in \mathbb{R}$.
		\end{enumerate}
	\end{assumption}
	\begin{remark}\label{remark J}
		Assumption $\J \in \W^{1,1}(\mathbb{R}^2;\R)$ can be weakened. Indeed, it can be replaced by $\J \in \W^{1,1}(\B_\delta;\R)$, where $\B_\delta := \{z \in \mathbb{R}^2 : |z| < \delta \}$ with $\delta := \text{diam}(\Omega)$, or also by
		\begin{eqnarray} \label{Estimate J}
		\sup_{x\in \Omega} \int_\Omega \left( |\J(x-y)| + |\nabla \J(x-y)| \right) dy < +\infty .
		\end{eqnarray}
	\end{remark}
	
	\begin{remark}\label{remark F}
		Since $\F(\cdot)$ is bounded from below, it is easy to see that Assumption \ref{prop of F and J} (4) implies that $\F(\cdot)$ has a polynomial growth of order $r'$, where $r' \in [2,\infty)$ is the conjugate index to $r$. Namely, there exist $C_5$ and $C_6 \geq 0$ such that 
		\begin{eqnarray}\label{2.9}
		|\F(s)| \leq C_5|s|^{r'} + C_6, \text{ for all } s \in \R.
		\end{eqnarray}
		Observe that Assumption \ref{prop of F and J} (4) is fulfilled by a potential of arbitrary polynomial growth. For example, Assumption (2)-(4) are satisfied for the case of the well-known double-well potential $\F (s)= (s^2 - 1)^2$.
	\end{remark}
	The existence of a weak solution for such a system in two and three dimensional bounded domains is established in Theorem 1, Corollaries 1 and 2, \cite{weak}, and the uniqueness for two dimensional case  is  obtained in Theorem 2, \cite{unique}. Under some extra assumptions on $\F$ and $\J$, and enough regularity on the initial data and external forcing,  a unique global strong solution is established in Theorem 2, \cite{strong}.
	\begin{definition}[weak solution]
		Let $\u_0\in\G_{\text{div}}$, $\varphi_0\in\H$ with $\F(\varphi_0)\in\mathrm{L}^1(\Omega)$ and $0<T<\infty$ be given. Then $(\u,\varphi)$ is a \emph{weak solution} to the uncontrolled system ($\U=0$) \eqref{nonlin phi}-\eqref{initial conditions} on $[0,T]$ corresponding to initial conditions $\u_0$ and $\varphi_0$ if 
		\begin{itemize}
			\item [(i)] $\u,\varphi$ and $\mu$ satisfy 
			\begin{equation}\label{sol}
			\left\{
			\begin{aligned}
			&	\u \in \mathrm{L}^{\infty}(0,T;\G_{\text{div}}) \cap \mathrm{L}^2(0,T;\V_{\text{div}}),  \\ 
			&	\u_t \in \mathrm{L}^{2-\gamma}(0,T;\V_{\text{div}}'), \ \text{ for all  }\gamma \in (0,1),  \\
			&	\varphi \in \mathrm{L}^{\infty}(0,T;\mathrm{H}) \cap \mathrm{L}^2(0,T;\mathrm{V}),   \\
			&	\varphi_t \in \mathrm{L}^{2-\delta}(0,T;\mathrm{V}'),\ \text{ for all  }\delta \in (0,1), \\
			&	\mu \in \mathrm{L}^2(0,T;\mathrm{V}),
			\end{aligned}
			\right.
			\end{equation}
			\item [(ii)]  for every $\psi\in\mathrm{V}$, every $\v \in \V_{\text{div}}$, if we define 		             $\rho$ by 
			\begin{align}
			\rho(x,\varphi):=a(x)\varphi+\F'(\varphi),
			\end{align} and for almost any $t\in(0,T)$, we have
			\begin{align}
			\langle \varphi_t,\psi\rangle +(\nabla\rho,\nabla\psi)&=\int_{\Omega}(\u\cdot\nabla\psi)\varphi\d x+\int_{\Omega}(\nabla\J*\varphi)\cdot\nabla\psi\d x,\\
			\langle \u_t,\v\rangle +\nu(\nabla\u,\nabla\v)+b(\u,\v,\w)&=-\int_{\Omega}(\v\cdot\nabla\mu)\varphi\ d x+\langle\f,\v\rangle.
			\end{align}
			\item [(iii)] Moreover, the following initial conditions hold in the weak sense 
			\begin{align}
			\u(0)=\u_0,\ \varphi(0)=\varphi_{0},
			\end{align}
			i.e., for every $\v \in \V_{\text{div}}$, we have $(\u(t),\v) \to (\u_0,\v)$ as $t\to 0$, and for every $\chi \in \mathrm{V}$, we have $(\varphi(t),\chi) \to (\varphi_0,\chi)$ as $t\to  0$.
		\end{itemize}
	\end{definition}
	
	\begin{theorem}[Existence, Theorem 1, Corollaries 1 and 2, \cite{weak}]\label{exist}
		Let the Assumption \ref{prop of F and J} be satisfied. Let $\u_0 \in \G_{\text{div}}$, $\varphi_0 \in \mathrm{H}$ be such that $\F(\varphi_0) \in \mathrm{L}^1(\Omega)$ and $\mathbf{f} \in \mathrm{L}^2_{\text{loc}}([0,\infty), \V_{\text{div}}')$ are given. Then, for every given $T>0$, there exists a weak solution $(u,\varphi)$ to the uncontrolled system  \eqref{nonlin phi}-\eqref{initial conditions} such that \eqref{sol} is satisfied.
		Furthermore, setting
		\begin{align}\mathscr{E}(\u(t),\varphi(t)) = \frac{1}{2} \|\u(t)\|^2 + \frac{1}{4} \int_\Omega \int_\Omega \J(x-y) (\varphi(x,t) - \varphi(y,t))^2 \d x \d y + \int_\Omega \F(\varphi(x,t))\d x,\end{align}
		the following energy estimate holds for almost any $t>0$:
		\begin{align}\label{energy}
		\mathscr{E}(\u(t),\varphi(t)) + \int_0^t \left(\nu \| \nabla \u(s)\|^2 + \| \nabla\mu(s) \|^2 \right)\d s \leq \mathscr{E}(\u_0,\varphi_0) + \int_0^t \langle \mathbf{f}(s), \u(s) \rangle\d s,\end{align}
		and the weak solution $(\u,\varphi)$ satisfies the following energy identity,
		$$\frac{\d}{\d t}\mathscr{E}(\u(t),\varphi(t)) + \nu \|\nabla \u(t) \|^2+ \| \nabla \mu(t) \|^2 = \langle \mathbf{f}(t) , \u(t) \rangle.$$
	\end{theorem}

	\begin{remark}\label{rem2.5}
		 We denote by $\mathbb{Q}$ a continuous monotone increasing function with respect to each of its arguments. As a consequence of energy inequality (\ref{energy}), we have the following bound:
		\begin{align}
		&	\|\u\|_{\mathrm{L}^{\infty}(0,T;\G_{\text{div}}) \cap \mathrm{L}^2(0,T;\V_{\text{div}})} + \|\varphi \|_{\mathrm{L}^{\infty}(0,T;\mathrm{H}) \cap \mathrm{L}^2(0,T;\mathrm{V})} + \|\F(\varphi)\|_{\mathrm{L}^{\infty}(0,T;\mathrm{H})} \nonumber\\&\qquad\leq \mathbb{Q}\left(\mathscr{E}(\u_0,\varphi_0),\|\mathbf{f}\|_{\mathrm{L}^2(0,T;\V_{\text{div}}')}\right),
		\end{align}
		where $\mathbb{Q}$ also depends on $\F$, $\J$, $\nu$ and $\Omega$. The above theorem also implies that $\u\in\C([0,T];\G_{\text{div}})$ and $\varphi\in \C([0,T];\mathrm{H})$.
	\end{remark}
	
	\begin{theorem}[Uniqueness, Theorem 2, \cite{unique}]\label{unique}
		Suppose that the Assumption \ref{prop of F and J} is satisfied. Let $\u_0 \in \G_{\text{div}}$, $\varphi_0 \in \mathrm{H}$ with $\F(\varphi_0) \in \mathrm{L}^1(\Omega)$ and $\mathbf{f} \in \mathrm{L}^2_{\text{loc}}([0,\infty);\V_{\text{div}}')$ be given. Then, the weak solution $(\u,\varphi)$ corresponding to $(\u_0,\varphi_0)$ and given by Theorem \ref{exist} is unique. 
		
	\end{theorem}
	
	For the further analysis of this paper, we also need the existence of a unique strong solution to the uncontrolled system  \eqref{nonlin phi}-\eqref{initial conditions}. The following theorem established in Theorem 2, \cite{strong} gives the existence and uniqueness of the strong solution  for the uncontrolled system  \eqref{nonlin phi}-\eqref{initial conditions}. 
	\begin{theorem}[Global Strong Solution, Theorem 2, \cite{strong}]\label{strongsol}
		Let $\mathbf{f}\in \mathrm{L}^2_{\text{loc}}([0,\infty);\G_{\text{div}})$, $\u_0\in\V_{\text{div}}$, $\varphi_0\in\mathrm{V}\cap\mathrm{L}^{\infty}(\Omega)$ be given and the Assumption \ref{prop of F and J} be satisfied. Then, for a given $T > 0$, there exists \emph{a unique weak solution} $(\u,\varphi)$ to the system \eqref{nonlin phi}-\eqref{initial conditions}(with $\U=0$) such that
		\begin{equation}\label{0.22}
		\left\{
		\begin{aligned}
		& \u\in \mathrm{L}^{\infty}(0,T;\V_{\text{div}})\cap \mathrm{L}^2(0,T;\H^2), \ \ \varphi\in \mathrm{L}^{\infty}((0,T) \times\Omega)\times \mathrm{L}^{\infty}(0,T;\mathrm{V}),\\ &\u_t\in\mathrm{L}^2(0,T;\G_{\text{div}}), \ \ \varphi_t\in\mathrm{L}^2(0,T;\mathrm{H}).
		\end{aligned}
		\right.
		\end{equation}
		Furthermore, suppose in addition that $\F \in\C^3(\mathbb{R}),\ a \in\mathrm{H}^2(\Omega)$ and that $\varphi_0\in \mathrm{H}^2(\Omega)$. Then, the uncontrolled system \eqref{nonlin phi}-\eqref{initial conditions} admits \emph{a unique strong solution} on $[0, T ]$ satisfying (\ref{0.22}) and also
		\begin{equation}
		\label{0.23}\left\{
		\begin{aligned}
		&\varphi \in\mathrm{L}^{\infty}(0,T;\mathrm{W}^{1,p}),\ 2\leq p<\infty,\\ &\varphi_t\in \mathrm{L}^{\infty}(0,T;\mathrm{H})\cap\mathrm{L}^2(0,T;\mathrm{V}).
		\end{aligned}
		\right.
		\end{equation}
		If $\J \in\mathrm{W}^{2,1}(\mathbb{R}^2;\R)$, we have in addition
		\begin{equation}\label{0.24}\varphi\in \mathrm{L}^{\infty}(0, T ; \mathrm{H}^2). \end{equation}
	\end{theorem}
	
	\begin{remark}\label{rem2.12}
		The regularity properties given in (\ref{0.22})-(\ref{0.24}) imply that
		\begin{align}\label{ues}
		\u\in\C  ([0, T] ; \V_{\text{div}} )), \ \varphi \in \C ( [0, T ]; \mathrm{V} )  \cap \C_w ( [0, T ]; \mathrm{H}^ 2),
		\end{align}
		where $\C_w$ means continuity on the
		interval $(0,T)$ with values in the weak topology of $\mathrm{H}^2$. Moreover, we have strong continuity in time, that is,	
		\begin{equation}
		\varphi\in\C([0,T];\mathrm{H}^2).
		\end{equation}
	\end{remark}

	\subsection{Existence and uniqueness of the linearized system}\label{se3}\setcounter{equation}{0} 
	Let us linearize the equations \eqref{nonlin phi}-\eqref{initial conditions}  around $(\widehat{\u}, \widehat{\varphi})$ which is the \emph{unique weak solution} of system (\ref{nonlin phi})-(\ref{initial conditions}) with control term $\U=0$ (uncontrolled system) and external forcing $\widehat{\mathbf{f}}$ such that
	$$\widehat{\mathbf{f}}\in \mathrm{L}^2(0,T;\G_{\text{div}}),\ \widehat{\u}_0\in\V_{\text{div}},\ \widehat{\varphi}_0\in\mathrm{V}\cap\mathrm{L}^{\infty}(\Omega),$$ so that $(\widehat{\u}, \widehat{\varphi})$ has the regularity given in (\ref{0.22}). We consider the following linearized system (see \cite{BDM}):
	\begin{subequations}
		\begin{align}
		\w_t - \nu \Delta \w + (\w \cdot \nabla )\widehat{\u} + (\widehat{\u} \cdot \nabla )\w + \nabla \widetilde{\uppi}_\w &= -\nabla a\psi \widehat{\varphi} -(\J\ast \psi) \nabla \widehat{\varphi} - (\J \ast \widehat{\varphi}) \nabla \psi + \widetilde{\mathbf{f}} +\U, \label{lin w} \\
		\psi_t + \w\cdot \nabla \widehat{\varphi} + \widehat{\u} \cdot\nabla \psi &= \Delta \widetilde{\mu}, \label{lin psi} \\ 
		\widetilde{\mu} &= a \psi - \J\ast \psi + \F''(\widehat{\varphi})\psi, \label{lin mu}\\
		\text{div }\w &= 0, \label{lin div zero}\\
		\frac{\partial \widetilde{\mu}}{\partial\mathbf{n}} &= 0, \ \w=0  \ \text{on } \ \partial \Omega \times (0,T), \label{lin boundary conditions}\\
		\w(0) &= \w_0, \ \psi(0) = \psi _0 \ \text{ in } \ \Omega, \label{lin initial conditions}
		\end{align} 
	\end{subequations}
	where $ \widetilde{\uppi}_\w = \widetilde{\uppi}-(\F'(\widehat{\varphi})+a\widehat{\varphi})\psi$. 
	Next, we discuss the unique global solvability results for the system (\ref{lin w})-(\ref{lin initial conditions}).

	\begin{theorem}[Existence and uniqueness of linearized system]\label{linearized}
		Suppose that the Assumption \ref{prop of F and J} is satisfied. Let us assume $(\widehat{\u},\widehat{\varphi})$ be the unique weak solution of the system \eqref{nonlin phi}-\eqref{initial conditions} with the regularity given in (\ref{0.22}). 
		
		Let $\w_0 \in \G_{\text{div}}$ and $\psi_0 \in \mathrm{V}'$ with $\widetilde{\mathbf{f}} \in \mathrm{L}^2(0,T;\G_{\text{div}})$, $\U \in \mathrm{L}^2(0,T;\V_{\text{div}}')$.  Then for a given $T >0$, there exists \emph{a unique weak solution} $(\w,\psi)$ to the system \eqref{lin w}-\eqref{lin initial conditions} such that
		$$\w \in \mathrm{L}^{\infty}(0,T;\G_{\text{div}}) \cap \mathrm{L}^{2}(0,T;\V_{\text{div}})\ \text{ and 
		} \ \psi \in \mathrm{L}^{\infty}(0,T;\mathrm{V}') \cap  \mathrm{L}^{2}(0,T;\mathrm{H}).$$
	\end{theorem}

	The regularity of the solution proved in Theorem \ref{linearized} is not sufficient for getting optimality condition. If we assume that $
	\F \in\C^3(\mathbb{R}),\ a \in\mathrm{H}^2(\Omega),$  then one can also prove that  $\psi\in\C([0,T];\mathrm{H})\cap\mathrm{L}^2(0,T;\mathrm{V})$.  Thus, we have 
	\begin{theorem}
		Suppose that the assumptions given in Theorem \ref{linearized} holds and assume that $
		\F \in\C^3(\mathbb{R}),\ a \in\mathrm{H}^2(\Omega).$
		Let $\w_0 \in \G_{\text{div}}$ and $\psi_0 \in \mathrm{V}'$ with $\widetilde{\mathbf{f}} \in \mathrm{L}^2(0,T;\G_{\text{div}})$, $\U \in \mathrm{L}^2(0,T;\V_{\text{div}}')$.  Then for a given $T >0$, there exists \emph{a unique weak solution} $(\w,\psi)$ to the system \eqref{lin w}-\eqref{lin initial conditions} such that
		$$\w \in \mathrm{L}^{\infty}(0,T;\G_{\text{div}}) \cap \mathrm{L}^{2}(0,T;\V_{\text{div}})\ \text{ and 
		} \ \psi \in \mathrm{L}^{\infty}(0,T;\mathrm{H}) \cap  \mathrm{L}^{2}(0,T;\mathrm{V}).$$
	\end{theorem}

	\section{Optimal Control Problem}\label{se4}\setcounter{equation}{0}In this section, we formulate a distributed optimal control problem  as the minimization of a suitable cost functional subject to the controlled nonlocal Cahn-Hilliard-Navier-Stokes equations. The main aim  is to establish a first and second order necessary and sufficient optimality condition for the existence of an optimal control that minimizes the cost functional given below subject to the constraint  \eqref{nonlin phi}-\eqref{initial conditions}.
	We consider the following cost functional: 
	
	\begin{align}
	\label{cost}
	\mathcal{J}(\u,\varphi,\U) := \int_0^T [ g(t,\u(t))+ h(t,\varphi(t)) + l(\U(t))] \d t,\end{align}
	To study the optimality condition we need to ensure the existence of strong solution to \eqref{nonlin phi}-\eqref{initial conditions}.  Hence,  let us assume that 
	\begin{align}\label{fes}
	\F \in\C^3(\mathbb{R}),\ a \in\mathrm{H}^2(\Omega),
	\end{align}
	and the initial data  
	\begin{align}\label{initial}
	\u_0\in\V_{\text{div}}\text{ and }\varphi_0\in \mathrm{H}^2(\Omega).
	\end{align}  
	The embedding of $\mathrm{V}$ and $\mathrm{L}^{\infty}(\Omega)$ in $\mathrm{H}^{2}(\Omega)$, implies $\varphi_0\in \mathrm{V}\cap\mathrm{L}^{\infty}(\Omega)$. From now onwards, we assume that along with the Assumption \ref{prop of F and J} conditions (\ref{fes}) and (\ref{initial}) hold true. 
	\begin{definition}
		Let $\mathscr{U}_{\text{ad}}$ be a closed and convex  subset  consisting of controls $\U \in\mathrm{L}^{2}(0,T;\G_{\text{div}})$. 
	\end{definition}
	
	\begin{definition}[Admissible Class]\label{definition 1}
		The \emph{admissible class} $\mathscr{A}_{\text{ad}}$ of triples $(\u,\varphi,\U)$ is defined as the set of states $(\u,\varphi)$ with initial data given in (\ref{initial}), solving the system \eqref{nonlin phi}-\eqref{initial conditions} with control $\U \in \mathscr{U}_{ad},$ which is  a subspace of $ \mathrm{L}^2(0,T;\G_{\text{div}})$.  That is,
		\begin{align*}
		\mathscr{A}_{\text{ad}}:=\Big\{(\u,\varphi,\U) :(\u,\varphi)\text{ is \emph{a unique strong solution} of }\eqref{nonlin phi}\text{-}\eqref{initial conditions}  \text{ with control }\U\Big\}.
		\end{align*}
	\end{definition}
	In view of the above definition, the optimal control problem we are considering is  formulated as
	\begin{align}\label{control problem}\tag{OCP}
	\min_{ (\u,\varphi,\U) \in \mathscr{A}_{\text{ad}}}  \mathcal{J}(\u,\varphi,\U).
	\end{align}
	We call $(\u,\varphi,\U)\in\mathscr{A}_{\text{ad}},$ a \emph{feasible triplet} for the problem \eqref{control problem}.
	\begin{definition}[Optimal Solution]
		A solution to the Problem \eqref{control problem} is called an \emph{optimal solution} and the optimal triplet is denoted by $(\u^* ,\varphi^*, \U^*)$. The control $\U^*$ is called an \emph{optimal control}.
	\end{definition}

			Let us now recall the definition of subgradient  of a function and its subdifferential, which is useful in the sequel. Let $\X$ be a real Banach space, $\X'$ be its topological dual and $\langle\cdot,\cdot\rangle_{\X'\times\X}$ be the duality pairing between $\X'$ and $\X$. 
	\begin{definition}[Subgradient, Subdifferential] Let $f:\X\to(-\infty,\infty],$  a functional on $\X$. A linear functional $u'\in\X'$  is called \emph{subgradient} of $f$ at $u$ if $f(u)\neq +\infty$ and for all $v \in\X$
		$$f(v)\geq f(u)+\langle u',v-u\rangle_{\X'\times \X},$$ holds. The set of all subgradients of $f$ at u is called \emph{subdifferential} $\partial f(u)$ of $f$ at $u$.
	\end{definition}
	We say that $f$ is \emph{G\^ateaux differentiable} at $u$ in $\X$ if $\partial f(u)$ consists of exactly one element, which we denote by $f_u(u)$.   This is equivalent to the assertion that the limit $$\langle f_u(u),h\rangle_{\X'\times\X}=\lim_{\tau\to 0}\frac{f(u+\tau h)-f(u)}{\tau}=\frac{\d}{\d\tau}f(u+\tau h)\bigg|_{\tau=0},$$ exists for all $h\in\X$.

	\begin{assumption}\label{ass 4.1} Let us assume that 
		\begin{enumerate}
			\item $g:[0,T] \times \G_{\text{div}} \rightarrow \mathbb{R}^+$ is measurable in the first variable, $g(t,0) \in \mathrm{L}^{\infty}(0,T)$, and for $\gamma_1 >0$, there exists $C_{\gamma_1}>0$ independent of $t$ such that
			\begin{align*}
			|g(t,\u_1) - g(t,\u_2)| \leq C_{\gamma_1}\|\u_1 - \u_2\|, \ \text{ for all } \ t\in [0,T], \|\u_1\| + \|\u_2\| \leq \gamma_1.
			\end{align*} 
			Moreover, G\^ateaux derivatives $g_{\u}(t,\cdot)$ and $g_{\u\u}(t.\cdot)$ is continuous in $\G_{\text{div}}$ for all $t\in [0,T].$
			\item $h:[0,T] \times \H \rightarrow \mathbb{R}^+$ is measurable in the first variable, $h(t,0) \in \mathrm{L}^{\infty}(0,T)$, and for $\gamma_2>0$, there exists $C_{\gamma_2}>0$ independent of $t$ such that
			\begin{align*}
			|h(t,\varphi_1) - h(t,\varphi_2)| \leq C_{\gamma_2}\|\varphi_1 - \varphi_2\|, \ \text{ for all } \ t\in [0,T], \|\varphi_1\| + \|\varphi_2\| \leq \gamma_2.
			\end{align*} 
			Further, G\^ateaux derivatives $h_{\varphi}(t,\cdot)$ and $h_{\varphi\varphi}(t.\cdot)$ is continuous in $\mathrm{H}$ for all $t\in [0,T].$
			\item $l:\mathbb{G}_{\text{div}}\rightarrow (-\infty,+\infty]$ is convex and lower semicontinuous, G\^ateaux derivatives $l_{\U}(\cdot)$ and $l_{\U \U}(\cdot)$ is continuous in $\mathbb{G}_{\text{div}}$. Moreover, there exist $C_7>0$ and $C_8 \in \mathbb{R}$ such that $$l(\U) \geq C_7\|\U\|^2 - C_8,  \ \text{ for all } \ \U \in \mathbb{G}_{\text{div}}.$$
		\end{enumerate}
	\end{assumption}

\begin{theorem}[Existence of an Optimal Triplet, Theorem 4.5, \cite{BDM}]\label{optimal}
	Let the Assumptions \ref{ass 4.1} and \ref{prop of F and J} along with the condition (\ref{fes}) holds true and the initial data $(\u_0,\varphi_0)$ satisfying (\ref{initial}) be given. Then there exists at least one triplet  $(\u^*,\varphi^*,\U^*)\in\mathscr{A}_{\text{ad}}$  such that the functional $ \mathcal{J}(\u,\varphi,\U)$ attains its minimum at $(\u^*,\varphi^*,\U^*)$, where $(\u^*,\varphi^*)$ is the unique strong solution of \eqref{nonlin phi}-\eqref{initial conditions}  with the control $\U^*$.
\end{theorem}
	
	\subsection{The adjoint system} We will characterize the optimal solution $(\u^*,\varphi^*,\U^*)$ to the problem (\ref{control problem}) via adjoint variables.
	The adjoint variables $(\p,\eta)$ satisfy the following adjoint system (see section 4.1 of \cite{BDM} for more details): 
	\begin{equation}\label{adj}
	\left\{
	\begin{aligned}
	-\p_t- \nu \Delta \p +(\p \cdot \nabla)\u + (\u \cdot \nabla)\p + \nabla \varphi \eta+\nabla q &=g_{\u}(t,\u),\\
	-\eta_t   +\J \ast (\p \cdot \nabla \varphi) - (\nabla{\J} \ast  \varphi)\cdot \p+ \nabla a\cdot \p \varphi 
	- \u \cdot\nabla \eta - a \Delta \eta&\\ + \J\ast \Delta \eta - \F''(\varphi)\Delta \eta &=h_{\varphi}(t,\varphi), \\ \text{div }\p&=0,\\ \p\big|_{\partial\Omega}=\eta\big|_{\partial \Omega}&=0,\\ \p(T,\cdot)=\eta(T,\cdot)&=0.
	\end{aligned}
	\right.
	\end{equation}
	Remember that $q:\Omega\to\R$ is also an unknown. We are considering the Dirichlet boundary condition for $\eta$  due to some technical difficulties (see Remark \ref{rem3.9}). The following theorem gives the unique solvability of the system (\ref{adj}). 
	\begin{theorem}[Existence and uniqueness of adjoint system, Theorem 4.4, \cite{BDM}]\label{adjoint}
		Let $(\u,\varphi)$ be a unique strong solution of the nonlinear system \eqref{nonlin phi}-\eqref{nonlin u}. Then, there exists \emph{a unique weak solution}\begin{align}\label{space}
		(\p,\eta)\in(\mathrm{C}([0,T];\G_{\text{div}})\cap\mathrm{L}^2(0,T;\V_{\text{div}}))\times (\mathrm{C}([0,T];\mathrm{H})\cap \mathrm{L}^2(0,T;\mathrm{V}))\end{align}  to the system \eqref{adj}. 
	\end{theorem}
	In Theorem 4.2, \cite{BDM2}, it has also been proved that $(\p,\eta)$ has the following regularity: 
	\begin{align}
	(\p,\eta)\in(\C([0,T];\V_{\text{div}})\cap\mathrm{L}^2(0,T;\H^2))\times (\C([0,T]; \mathrm{V})\cap \mathrm{L}^2(0,T;\mathrm{H}^2)).\end{align}

	\subsection{Pontryagin's maximum principle}
	In this subsection, we state the Pontryagin's maximum principle  for the  optimal control problem defined in \eqref{control problem}. Pontryagin Maximum principle  gives a first order necessary condition for the optimal control problem  \eqref{control problem}, which also characterize the optimal control in terms of the adjoint variables.

	The following minimum principle is satisfied by the optimal triplet $(\u^*,\varphi^*,\U^*)\in\mathscr{A}_{\text{ad}}$:
	\begin{eqnarray}
	l(\U^*(t))+( \p(t),\U^*(t) ) \leq l(\W)+( \p(t),\W ),
	\end{eqnarray} 
	for all  $\W\in\G_{\text{div}},$ and a.e. $t\in[0,T]$. 
	Equivalently the above minimum principle may be  written in terms of the \emph{Hamiltonian formulation}. 
	Let us first define the \emph{Lagrangian} by
	$$\mathscr{L}(\u,\varphi,\U) := g(t,\u(t))+ h(t,\varphi(t)) + l(\U(t)).$$
	We define the corresponding \emph{Hamiltonian} by
	$$\mathscr{H}(\u,\varphi,\U,\p,\eta):= \mathscr{L}(\u,\varphi,\U) +\langle \p, \mathscr{N}_1(\u,\varphi,\U) \rangle + \langle \eta, \mathscr{N}_2(\u,\varphi)\rangle,$$ where $\mathscr{N}_1$ and $\mathscr{N}_2$ are defined by 
	\begin{equation}\label{n1n2}
	\left\{
	\begin{aligned} 
	\mathscr{N}_1(\u,\varphi, \U) &:= \nu \Delta \u - (\u \cdot \nabla)\u - \nabla \widetilde{\uppi } -(\J\ast \varphi) \nabla \varphi -\nabla a\frac{\varphi^2}{2} +\U,  \\
	\mathscr{N}_2(\u,\varphi) &:= -\u\cdot \nabla \varphi + \Delta (a\varphi -\J\ast \varphi + \F'(\varphi)).
	\end{aligned}
	\right.
	\end{equation}
	Hence, we get the minimum principle as
	\begin{eqnarray}
	\mathscr{H}(\u^*(t),\varphi^*(t),\U^*(t),\p(t),\eta(t)) \leq \mathscr{H}(\u^*(t),\varphi^*(t),\mathrm{W},\p(t),\eta(t)),
	\end{eqnarray}
	for all $\mathrm{W} \in \G_{\text{div}}$, a.e., $t\in[0,T]$.

	\begin{theorem}[Pontryagin's minimum principle]\label{main}
		Let $(\u^*,\varphi^*,\U^*)\in\mathscr{A}_{\text{ad}}$ be the optimal solution of the Problem \ref{control problem}. Then there exists a unique weak solution $(\p,\eta)$ of the adjoint system \eqref{adj} and for almost every $t \in [0,T]$ and $\mathrm{W} \in \G_{\text{div}}$, we have 
		\begin{eqnarray}\label{3.41}
		l(\U^*(t))+( \p(t),\U^*(t) ) \leq l(\W)+( \p(t),\W ).  
		\end{eqnarray}
	\end{theorem}
	Furthermore, we can write \eqref{3.41} as
	\begin{eqnarray}
	( -\p(t), \mathrm{W} -\U^*(t)) \leq l(\W) -  l(\U^*(t)),
	\end{eqnarray}
	a.e. $t \in [0,T]\text{ and }\mathrm{W} \in \G_{\text{div}}$. In the above form, we see that $-\p \in \partial  l(\U^*(t))$, where $\partial$ denotes the subdifferential. Whenever $l$ is G\^ateaux differentiable, it follows that
	\begin{equation}\label{nec}
	-\p(t) = l_{\U}(\U^*(t)), \ \text{ a.e. }\  t \in [0,T]. 
	\end{equation}
	If $l$ is not G\^ateaux differentiable then, we have $-\p \in \partial  l(\U^*(t))$. For a proof of the above theorem (first order necessary condition), see Theorem 4.7, \cite{BDM}.
	\subsection{Second order necessary and sufficient optimality condition}
	In this subsection we derive the second order necessary and sufficient optimality condition for the optimal control problem \eqref{control problem}. 
	
	Let $(\widehat{\u},\widehat{\varphi},\widehat{\U})$ be an arbitrary feasible triplet for the optimal control problem \eqref{control problem}, we set
	\begin{align}\label{def Q}
	\mathcal{Q}_{\widehat{\u},\widehat{\varphi},\widehat{\U}}= \{(\u,\varphi,\U) \in \mathcal{A}_{\text{ad}} \} - \{(\widehat{\u},\widehat{\varphi},\widehat{\U})\},
	\end{align}
	which denotes the differences of all feasible triplets for the problem \eqref{control problem} corresponding to $(\widehat{\u},\widehat{\varphi},\widehat{\U})$.
	\begin{theorem}[Necessary condition]\label{necessary}
		Let  $(\u^*,\varphi^*,\U^*)$ be an optimal triplet for the problem \eqref{control problem}. Suppose the Assumption \ref{ass 4.1} holds and the adjoint variables $(\p,\eta)$ satisfies the adjoint system \eqref{adj}. Then for any $(\u,\varphi,\U) \in \mathcal{Q}_{\u^*,\varphi^*,\U^*}$, there exist $0\leq \theta_1,\theta_2,\theta_3,\theta_4\leq1$ such that
		\begin{align}\label{3.14}
		&\int_0^T (g_{\u\u}(t,\u^*+\theta_1\u)\u,\u) \d t + \int_0^T ( h_{\varphi\varphi}(t,\varphi^*+\theta_2\varphi)\varphi,\varphi) \d t + \int_0^T ( l_{\U\U}(t,\U^*+\theta_3\U)\U,\U) \d t \nonumber \\ 
		&\quad-2\int_0^T \big( (\u \cdot \nabla )\u+ \frac{\nabla a}{2} \varphi^2 +(\J\ast \varphi) \nabla \varphi, \p \big) \d t - 2\int_0^T (\u \cdot \nabla \varphi,\eta ) \d t \nonumber \\
		&\quad+\int_0^T \big((\F'''(\varphi^*+ \theta_4 \varphi)\varphi^2,\Delta\eta\big) \d t \geq 0.
		\end{align}
	\end{theorem}
	\begin{proof}
		For any $(\u,\varphi,\U)\in \mathcal{Q}_{\u^*,\varphi^*,\U^*},$ by  \eqref{def Q} there exist $(\z,\xi,\W) \in \mathcal{A}_{\text{ad}}$ such that $(\u,\varphi,\U)= (\z - \u^*,\xi - \varphi^*,\W - \U^*)$. So from \eqref{nonlin phi}-\eqref{initial conditions}, we can derive that $(\u,\varphi)$ satisfies the following system:
		\begin{eqnarray}\label{5.29}
		\left\{
		\begin{aligned}
		{\u}_t - \nu \Delta \u + (\u \cdot \nabla )\u^* &+ (\u^* \cdot \nabla )\u +(\J\ast \varphi) \nabla \varphi^*+ (\J\ast \varphi^*) \nabla \varphi+\nabla a \varphi \varphi^* + \nabla \widetilde{\uppi}_{\u} \\&= -(\u \cdot \nabla )\u- \frac{\nabla a}{2} \varphi^2 -(\J\ast \varphi) \nabla \varphi + \U, \ \text{ in } \ \Omega\times(0,T),\\
		{\varphi}_t + \u \cdot \nabla \varphi + \u^* \cdot\nabla \varphi + \u \cdot \nabla \varphi^* &= \Delta \widetilde{\mu}, \ \text{ in } \ \Omega\times(0,T), \\ 
		\widetilde{\mu} &= a \varphi - \J\ast \varphi +\F'(\varphi^*+  \varphi) - \F'(\varphi^*) ,\ \text{ in } \ \Omega\times(0,T),  \\
		\text{div }\u &= 0, \ \text{ in } \ \Omega\times(0,T), \\
		\frac{\partial \widetilde{\mu}}{\partial\mathbf{n}} &= 0, \ \u=0,  \ \text{on } \ \partial \Omega \times (0,T),\\
		\u(0) &= 0, \ \varphi(0) = 0, \ \text{ in } \ \Omega.
		\end{aligned}
		\right.
		\end{eqnarray}
		From now on we use the Taylor series expansion of $\mathrm{F}'$ around $\varphi^*$. There exists a $\theta_4$;  $0<\theta_4<1$ such that 
		\begin{align}\label{3.15}
		\F'(\varphi^*+  \varphi) - \F'(\varphi^*) =\F''(\varphi^*)  \varphi + \F'''(\varphi^*+ \theta_4 \varphi)\frac{\varphi^2}{2}.
		\end{align}
		Taking inner product of \eqref{5.29} with $(\p,\eta)$, integrating over $[0,T]$ and then adding, we get
		\begin{align}
		&\int_0^T ({\u}_t - \nu \Delta \u + (\u \cdot \nabla )\u^* + (\u^* \cdot \nabla )\u +(\J\ast \varphi) \nabla \varphi^*+ (\J\ast \varphi^*) \nabla \varphi+\nabla a \varphi \varphi^* + \nabla \widetilde{\uppi}_{\u},\p ) \d t \nonumber\\  & \quad+ \int_0^T ( (\u \cdot \nabla )\u+ \frac{\nabla a}{2} \varphi^2 +(\J\ast \varphi) \nabla \varphi - \U, \p ) \d t\nonumber\\&\quad+ \int_0^T ( {\varphi}_t + \u \cdot \nabla \varphi + \u^* \cdot\nabla \varphi + \u \cdot \nabla \varphi^* - \Delta \widetilde{\mu},\eta ) \d t =0 .
		\end{align}
		Using an integration by parts, we further get
		\begin{align}\label{3.18}
		&\int_0^T ( \u,-\p_t- \nu \Delta \p +(\p \cdot \nabla)\u^* + (\u^* \cdot \nabla)\p + \nabla \varphi^* \eta+\nabla q ) \d t\nonumber\\& + \int_0^T ( \varphi, -\eta_t  +\J \ast (\p \cdot \nabla \varphi^*) - (\nabla{\J} \ast  \varphi^*)\cdot \p + \nabla a\cdot \p \varphi^* - \u^* \cdot\nabla \eta)\d t\nonumber\\& + \int_0^T ( \varphi,  - a \Delta \eta + \J\ast \Delta \eta - \F''(\varphi^*)\Delta \eta )\d t+ \int_0^T (\u \cdot \nabla \varphi,\eta ) \d t\nonumber\\&+ \int_0^T ( (\u \cdot \nabla )\u+ \frac{\nabla a}{2} \varphi^2 +(\J\ast \varphi) \nabla \varphi - \U, \p ) \d t  -\int_0^T ((\F'''(\varphi^*+ \theta_4 \varphi)\frac{\varphi^2}{2},\Delta\eta) \d t =0.
		\end{align}
		Since $(\u^*,\varphi^*,\U^*)$ is an optimal triplet, it satisfies the first order necessary conditions given in \eqref{nec}. This and the adjoint system \eqref{adj} implies 
		\begin{align}\label{5.32}
		& \int_0^T ( l_{\U}(\U^*),\U) \d t +\int_0^T ( g_{\u}(t,\u^*),\u ) \d t+ \int_0^T ( h_{\varphi}(t,\varphi^*),\varphi ) \d t \nonumber\\&\quad
		+\int_0^T ( (\u \cdot \nabla )\u+ \frac{\nabla a}{2} \varphi^2 +(\J\ast \varphi) \nabla \varphi , \p ) \d t + \int_0^T (\u \cdot \nabla \varphi,\eta ) \d t \nonumber \\
		& \quad-\int_0^T ((\F'''(\varphi^*+ \theta_4 \varphi)\frac{\varphi^2}{2},\Delta\eta) \d t=0.
		\end{align}
		Since $(\u,\varphi,\U) \in \mathcal{Q}_{\u^*,\varphi^*,\U^*}$, by \eqref{def Q}, we have $(\u+\u^*,\varphi+\varphi^*,\U+\U^*)$ is a feasible triplet for the problem \eqref{control problem}. We obtain that
		\begin{align*}
		&\mathcal{J}(\u+\u^*,\varphi+\varphi^*,\U+\U^*) - \mathcal{J}(\u^*,\varphi^*,\U^*) \nonumber\\&= \int_0^T [ g(t,\u+\u^*)+ h(t,\varphi+\varphi^*) + l(\U+\U^*)] \d t \nonumber \\
		&\quad   - \int_0^T [ g(t,\u^*)+ h(t,\varphi^*) + l(\U^*)] \d t \geq 0.
		\end{align*} 
		Using the Taylor series expansion and Assumption \ref{ass 4.1} for the cost functional, there exist $0\leq \theta_1,\theta_2,\theta_3\leq 1$ such that 
		\begin{align}
		&\int_0^T \left[( g_{\u}(t,\u^*),\u ) + \frac{1}{2}( g_{\u\u}(t,\u^*+\theta_1\u)\u,\u)\right] \d t  \nonumber\\&\quad+ \int_0^T \left[( h_{\varphi}(t,\varphi^*),\varphi )+ \frac{1}{2}( h_{\varphi\varphi}(t,\varphi^*+\theta_2\varphi)\varphi,\varphi)\right] \d t  \nonumber \\
		&\quad +\int_0^T \left[( l_{\U}(\U^*),\U ) + \frac{1}{2}( l_{\U\U}(\U^*+\theta_3\U)\U,\U)\right] \d t \geq 0.
		\end{align}
		From \eqref{5.32}, it follows that
		\begin{align*}
		&\int_0^T \frac{1}{2}( g_{\u\u}(t,\u^*+\theta_1\u)\u,\u) \d t + \int_0^T  \frac{1}{2}( h_{\varphi\varphi}(t,\varphi^*+\theta_2\varphi)\varphi,\varphi) \d t  
		\int_0^T  \frac{1}{2}( l_{\U\U}(\U^*+\theta_3\U)\U,\U) \d t \nonumber \\ 
		& \quad-\int_0^T ( (\u \cdot \nabla )\u+ \frac{\nabla a}{2} \varphi^2 +(\J\ast \varphi) \nabla \varphi , \p ) \d t - \int_0^T (\u \cdot \nabla \varphi,\eta ) \d t \nonumber\\&\quad+\int_0^T ((\F'''(\varphi^*+ \theta_4 \varphi)\frac{\varphi^2}{2},\Delta\eta) \d t \geq 0,
		\end{align*}
		which completes the proof.
	\end{proof}
	\begin{remark}\label{rem3.9}
		Note that we assumed Dirichlet boundary condition on $\eta$ in \eqref{adj}. One can perform integration by parts in (\ref{3.18}) and obtain  the term $(\F'''(\varphi^*+ \theta_4 \varphi)\frac{\varphi^2}{2},\Delta\eta)$, which is well-defined. 
	\end{remark}
	
	\begin{theorem}[Sufficient condition]\label{sufficient}
		Suppose that Assumption \ref{ass 4.1} holds and $(\u^*,\varphi^*,\U^*)$ be a feasible triplet for the problem \eqref{control problem}. Let us assume that the first order necessary condition holds true (see  \eqref{nec}), and for any $0\leq \theta_1,\theta_2,\theta_3,\theta_4\leq 1$ and $(\u,\varphi,\U) \in \mathcal{Q}_{\u^*,\varphi^*,\U^*},$ the following inequality holds:
		\begin{align}\label{3.21}
		&\int_0^T (g_{\u\u}(t,\u^*+\theta_1\u)\u,\u) \d t + \int_0^T ( h_{\varphi\varphi}(t,\varphi^*+\theta_2\varphi)\varphi,\varphi) \d t + \int_0^T ( l_{\U\U}(t,\U^*+\theta_3\U)\U,\U) \d t \nonumber \\ 
		&\quad-2\int_0^T \big( (\u \cdot \nabla )\u+ \frac{\nabla a}{2} \varphi^2 +(\J\ast \varphi) \nabla \varphi, \p \big) \d t -2\int_0^T (\u \cdot \nabla \varphi,\eta ) \d t \nonumber \\
		&\quad+\int_0^T \big((\F'''(\varphi^*+ \theta_4 \varphi)\varphi^2,\Delta\eta\big) \d t \geq 0.
		\end{align}
		Then $(\u^*,\varphi^*,\U^*)$ is an optimal triplet for the problem \eqref{control problem}.
	\end{theorem}
	\begin{proof}
		For any $(\z,\xi,\W) \in \mathcal{A}_{\text{ad}}$, we have by \eqref{def Q} that, $(\z - \u^*,\xi - \varphi^*,\W - \U^*) \in \mathcal{Q}_{\u^*,\varphi^*,\U^*}$ and  it satisfies:
		\begin{eqnarray}\label{5.34}
		\left\{
		\begin{aligned}
		({\z - \u^*})_t &- \nu \Delta (\z - \u^*) + ((\z - \u^*) \cdot \nabla )\u^* + (\u^* \cdot \nabla )(\z - \u^*) +(\J\ast (\xi - \varphi^*)) \nabla \varphi^*   \\
		&\quad+ (\J\ast \varphi^*) \nabla (\xi - \varphi^*) +\nabla a (\xi - \varphi^*) \varphi^*+ \nabla \widetilde{\uppi}_{(\z - \u^*)} \\&= -((\z - \u^*) \cdot \nabla )(\z - \u^*)- \frac{\nabla a}{2} (\xi - \varphi^*)^2-(\J\ast (\xi - \varphi^*)) \nabla (\xi - \varphi^*) \\&\quad+ \W -\U^*, \ \text{ in }\ \Omega\times(0,T),\\
		{(\xi - \varphi^*)}_t &+ (\z - \u^*) \cdot \nabla (\xi - \varphi^*) + \u^* \cdot\nabla (\xi - \varphi^*) + (\z - \u^*) \cdot \nabla \varphi^* \\&= \Delta \widetilde{\mu},\ \text{ in }\ \Omega\times(0,T), \\ 
		\widetilde{\mu} &= a (\xi - \varphi^*) - \J\ast (\xi - \varphi^*) +\mathrm{F}'(\xi)-\mathrm{F}'(\varphi^*),\\
		\text{div }(\z - \u^*) &= 0,\ \text{ in }\ \Omega\times(0,T), \\
		\frac{\partial \widetilde{\mu}}{\partial\mathbf{n}} &= 0, \ (\z - \u^*)=0,  \ \text{ on } \ \partial \Omega \times (0,T),\\
		(\z - \u^*)(0) &= 0, \ (\xi - \varphi^*)(0) = 0, \ \text{ in } \ \Omega.
		\end{aligned}
		\right.
		\end{eqnarray}
		As we argued in \eqref{3.15}, there exists a $0\leq \widetilde{\theta}_4\leq 1$ such that 
		\begin{align*}
		\mathrm{F}'(\xi)-\mathrm{F}'(\varphi^*)= \F''(\varphi^*) (\xi - \varphi^*) + \F'''(\varphi^*+\widetilde{\theta}_4(\xi - \varphi^*)) \frac{(\xi - \varphi^*)^2}{2}.
		\end{align*}
		Now multiplying \eqref{5.34} with $(\p,\eta)$, integrating over $[0,T]$ and then adding, we get  
		\begin{align}
		& \int_0^T ( \z-\u^*,-\p_t- \nu \Delta \p +(\p \cdot \nabla)\u^* + (\u^* \cdot \nabla)\p + \nabla \varphi^* \eta+\nabla q ) \d t \nonumber\\&
		\quad+ \int_0^T ( \xi-\varphi^*, -\eta_t   +\J \ast (\p \cdot \nabla \varphi^*) - (\nabla{\J} \ast  \varphi^*)\cdot \p + \nabla a\cdot \p \varphi^* - \u^* \cdot\nabla \eta - a \Delta \eta)\d t\nonumber\\&\quad  + \int_0^T ( \xi-\varphi^*,\J\ast \Delta \eta - \F''(\varphi^*)\Delta \eta) \d t 
		\nonumber\\&\quad +
		\int_0^T ( ((\z - \u^*) \cdot \nabla )(\z - \u^*) + \frac{\nabla a}{2} (\xi - \varphi^*)^2+(\J\ast (\xi - \varphi^*)) \nabla (\xi - \varphi^*) -\W + \U^*, \p ) \d t \nonumber \\
		&\quad + \int_0^T ((\z - \u^*) \cdot \nabla (\xi - \varphi^*),\eta) \d t  - \int_0^T((\F'''(\varphi^*+\widetilde{\theta}_4(\xi - \varphi^*)) \frac{(\xi - \varphi^*)^2}{2}),\Delta\eta)\d t =0,
		\end{align}
		where we also performed an integration by parts. Since $(\u^*,\varphi^*,\U^*)$ satisfies the first order necessary condition, i.e., $-\p(t) = l_{\U}(\U^*(t)),  \text{ a.e. } t \in [0,T]$  and using the adjoint system \eqref{adj}, we further get
		\begin{align}\label{5.35}
		& \int_0^T ( \z-\u^*,g_{\u}(\u^*) ) \d t + \int_0^T ( \xi-\varphi^*, h_{\varphi}(\varphi^*) ) \d t + \int_0^T ( l_{\U}(\U^*), \W - \U^* ) \d t \nonumber\\
		&\quad+\int_0^T ( ((\z - \u^*) \cdot \nabla )(\z - \u^*) + \frac{\nabla a}{2} (\xi - \varphi^*)^2+(\J\ast (\xi - \varphi^*)) \nabla (\xi - \varphi^*) , \p ) \d t \nonumber \\
		&\quad + \int_0^T ((\z - \u^*) \cdot \nabla (\xi - \varphi^*),\eta ) \d t -\int_0^T((\F'''(\varphi^*+\widetilde{\theta}_4(\xi - \varphi^*)) \frac{(\xi - \varphi^*)^2}{2}),\Delta\eta)\d t= 0.
		\end{align} 
		Using the Assumption \ref{ass 4.1} and Taylor's formula, we get
		\begin{align}\label{3.25}
		&\int_0^T \left[ g(t,\z(t))+ h(t,\xi(t)) + l(\W(t))\right] \d t -  \int_0^T \left[ g(t,\u^*(t))+ h(t,\varphi^*(t)) + l(\U^*(t))\right] \d t\nonumber \\
		&= \int_0^T \left[( g_{\u}(t,\u^*),\z-\u^* ) + \frac{1}{2}( g_{\u\u}(t,\u^*+\widetilde{\theta_1}(\z-\u^*))(\z-\u^*),(\z-\u^*))\right] \d t\nonumber \\
		&\quad + \int_0^T \left[( h_{\varphi}(t,\varphi^*),\xi -\varphi^* ) + \frac{1}{2}( h_{\varphi\varphi}(t,\varphi^*+\widetilde{\theta_2}(\xi -\varphi^*))(\xi -\varphi^*),(\xi -\varphi^*))\right] \d t  \nonumber\\
		&\quad +\int_0^T \left[( l_{\U}(t,\U^*),\W -\U^*) + \frac{1}{2}( l_{\U\U}(t,\U^*+\widetilde{\theta_3}(\W -\U^*))(\W -\U^*),\W -\U^*)\right] \d t, 
		\end{align}
		for some $0\leq \widetilde{\theta}_1,\widetilde{\theta}_2,\widetilde{\theta}_3\leq 1$. Using \eqref{5.35} in \eqref{3.25}, we also obtain
		\begin{align*}
		&\int_0^T \left[ g(t,\z(t))+ h(t,\xi(t)) + l(\W(t))\right] \d t -  \int_0^T \left[ g(t,\u^*(t))+ h(t,\varphi^*(t)) + l(\U^*(t))\right] \d t \\
		&=- \int_0^T ( ((\z - \u^*) \cdot \nabla )(\z - \u^*) + \frac{\nabla a}{2} (\xi - \varphi^*)^2+(\J\ast (\xi - \varphi^*)) \nabla (\xi - \varphi^*) , \p ) \d t \\
		&\quad - \int_0^T ((\z - \u^*) \cdot \nabla (\xi - \varphi^*),\eta ) \d t + \int_0^T((\F'''(\varphi^*+\widetilde{\theta}_4(\xi - \varphi^*)) \frac{(\xi - \varphi^*)^2}{2}),\Delta\eta)\d t \\
		&\quad+ \frac{1}{2}( g_{\u\u}(t,\u^*+\widetilde{\theta_1}(\z-\u^*))(\z-\u^*),(\z-\u^*))] \d t \\
		&\quad + \frac{1}{2}( h_{\varphi\varphi}(t,\varphi^*+\widetilde{\theta_2}(\xi -\varphi^*))(\xi -\varphi^*),(\xi -\varphi^*))] \d t \\
		&\quad + \frac{1}{2}( l_{\U\U}(t,\U^*+\widetilde{\theta_3}(\W -\U^*))(\W -\U^*),\W -\U^*)] \d t \geq 0,
		\end{align*}
		using \eqref{3.21}. Therefore we get for any $(\z,\xi,\W) \in \mathcal{A}_{\text{ad}}$, the following inequality holds:
		$$\int_0^T [ g(t,\z(t))+ h(t,\xi(t)) + l(\W(t))] \d t \geq  \int_0^T [ g(t,\u^*(t))+ h(t,\varphi^*(t)) + l(\U^*(t))] \d t,$$
		which implies that the triplet $(\u^*,\varphi^*,\U^*)$ is an optimal triplet for the problem \eqref{control problem}. 
	\end{proof}
	\begin{example}
		Let us consider the cost functional considered in \cite{BDM}, given by
		\begin{align*}
		\mathcal{J}(\u,\varphi,\U) := \frac{1}{2} \int_0^T \|\u(t)\|^2 \d t+ \frac{1}{2} \int_0^T \|\varphi(t)\|^2 \d t+ \frac{1}{2} \int_0^T\|\U(t)\|^2\d t.\end{align*}
		Let us assume $(\u^*,\varphi^*,\U^*)$ is an optimal solution for \eqref{control problem}. The Pontryagin's maxixmum principle gives
		$$\U^*(t) = -\p(t), \ \text{ for a.e. } \ t\in [0,T].$$
		Note that the Assumption \ref{ass 4.1} holds true and the adjoint variable $(\p,\eta)$ satisfies the adjoint system \eqref{adj} with $g_{\u}(t,\u)=\u$ and $h_{\varphi}(t,\varphi)=\varphi$. Let us take $\mathrm{F}(s)=(s^2-1)^2$. Then for any $(\u,\varphi,\U) \in \mathcal{Q}_{\u^*,\varphi^*,\U^*}$, there exists $0<\theta<1$ such that
		\begin{align*}
		& \int_0^T \|\u(t)\|^2 \d t+  \int_0^T \|\varphi(t)\|^2 \d t+ \int_0^T\|\U(t)\|^2\d t +24 \int_0^T((\varphi^*+\theta \varphi)\varphi^2,\Delta\eta)\d t  \\ 
		&\quad-2\int_0^T ( (\u \cdot \nabla )\u+ \frac{\nabla a}{2} \varphi^2 +(\J\ast \varphi) \nabla \varphi - \U, \p ) \d t - 2\int_0^T (\u \cdot \nabla \varphi,\eta ) \d t\geq 0.
		\end{align*}
		If $\J \in\mathrm{W}^{2,1}(\mathbb{R}^2;\R),$ then $\varphi,\varphi^*\in\mathrm{C}([0,T];\mathrm{H}^2)$ and $(\Delta((\varphi^*+\theta \varphi)\varphi^2),\eta),$ is well-defined. In this case, one can work with the Neumann boundary condition for $\eta$, i.e, $\frac{\partial\eta}{\partial \mathbf{n}}=0$, on $\partial\Omega\times(0,T)$ in \eqref{adj}.
	\end{example}


\end{document}